\documentclass[a4paper,11pt,reqno]{amsart}
\usepackage{amsfonts,amsmath,amsthm,amssymb,enumerate,dsfont,bbm,mathrsfs,verbatim}
\usepackage[british]{babel}

\newtheorem{theorem}{Theorem}[section]
\newtheorem{lemma}[theorem]{Lemma}
\newtheorem{corollary}[theorem]{Corollary}
\newtheorem{conjecture}[theorem]{Conjecture}

\newtheorem*{claim*}{Claim}
\newtheorem*{dilworth}{Dilworth's theorem}
\newtheorem*{backward_induction}{Backward induction}

\theoremstyle{definition}

\newtheorem*{alg}{Algorithm}

\numberwithin{equation}{section}

\newcommand{\Prob}{\mathbb{P}}
\newcommand{\Exp}{\mathbb{E}}
\newcommand{\Alg}{\mathcal{F}}
\newcommand{\Algg}{\mathcal{G}}
\newcommand{\Alggg}{\mathcal{H}}
\newcommand{\AlgA}{\mathcal{A}}
\newcommand{\Borel}{\mathcal{B}}
\newcommand{\Pow}{\mathcal{P}}
\newcommand{\Real}{\mathbb{R}}

\newcommand{\Class}{\mathcal{C}}
\newcommand{\Pos}{\mathscr{P}}

\newcommand{\dd}[2]{\frac{\textrm{d}#1}{\textrm{d}#2}}
\newcommand{\st}{$^{\text{st}}$}

\renewcommand{\th}{$^{\text{th}}$}
\renewcommand{\epsilon}{\varepsilon}

\begin{document}

\title{The secretary problem on an unknown poset}

\author{Bryn Garrod}
\address{Department of Pure Mathematics and Mathematical Statistics,
Centre for Mathematical Sciences,
University of Cambridge,
Wilberforce Road,
Cambridge,
CB3 0WB,
UK}
\email{b.garrod@dpmms.cam.ac.uk}
\thanks{}

\author{Robert Morris}
\address{IMPA, Estrada Dona Castorina 110, Jardim Bot\^anico, Rio de Janeiro, RJ, Brasil}
\email{rob@impa.br}
\thanks{This research was supported by: (BG) EPSRC; (RM) ERC Advanced grant DMMCA, and a Research Fellowship from Murray Edwards College, Cambridge}

%\date{\today}

\begin{abstract}
We consider generalizations of the classical secretary problem, also known as the problem of optimal choice, to posets where the only information we have is the size of the poset and the number of maximal elements. We show that, given this information,  there is an algorithm that is successful with probability at least $\frac{1}{e}$.  We conjecture that if there are $k$ maximal elements and $k \geq 2$ then this can be improved to $\sqrt[k-1]{\frac{1}{k}}$, and prove this conjecture for posets of width $k$.  We also show that no better bound is possible.
\end{abstract}

\maketitle

\section{Introduction} \label{sec:intro}

The exact origins of the classical secretary problem are complicated (and the subject of Ferguson's history of the problem~\cite{Fer}), but the problem was popularized by Martin Gardner in his \emph{Scientific American} column in February 1960, as the game \emph{googol}.  The problem itself is simple to state, and its `secretary problem' formulation is as follows.  There are $n$ candidates to be interviewed for a position as a secretary.  They are interviewed one by one and, after each interview, the interviewer must decide whether or not to accept that candidate.  If the candidate is accepted then the process stops, and if the candidate is rejected then the interviewer moves on to the next candidate.  The interviewer may only accept the most recently interviewed candidate.  At each stage, the interviewer knows the complete ranking of the candidates interviewed so far, all of whom are comparable, but has no other measure of their ability.  The interviewer is only interested in finding the very best candidate; selecting any other for the job is considered a failure.  It is well-known (see~\cite{Fer}, for example) that the interviewer has a simple strategy that is successful with probability at least $\frac{1}{e}$, and that there is no strategy achieving a better bound.

Since 1960, many generalizations of the problem have been considered.  One direction has been to consider partial orders on the candidates other than a total order.  In this case, the interviewer knows the poset induced by the candidates interviewed so far, and wishes to choose a candidate who is maximal in the original poset.  Morayne~\cite{Mor} considered the case of a full binary tree of depth $n$, and showed that the optimal strategy is to select the maximum out of the elements seen so far when the poset induced by these elements is either linear of length greater than $\frac{n}{2}$ or non-linear with a unique maximum.  He showed that as $n$ tends to infinity, the probability of success tends to 1.  Garrod, Kubicki and Morayne~\cite{GKM} considered the case of $n$ pairs of `twins', where there are $n$ levels with two incomparable elements on each level.  They showed that the optimal strategy is to wait until a certain threshold number of levels have been seen and then to select the next element that is maximal and whose twin has already been seen.  They further showed that as $n$ tends to infinity, this threshold tends to $\sim0.4709n$ and the probability of success to $\sim0.7680$.  Calculating these asymptotic values for the natural extension to `$k$-tuplets', for $k > 2$, seems to be a harder problem.

A further interesting generalization was an attempt to find an algorithm that was successful on \emph{all} posets of a given size with positive probability. Surprisingly, Preater~\cite{Pre} proved that there is such a `universal' algorithm (depending only on the size of the poset), which is successful on every poset with probability at least $\frac{1}{8}$.  In this algorithm, an initial random number of elements are rejected and a subsequent element is accepted according to randomized criteria.  A slightly modified version of the algorithm, also suggested by Preater, was analysed by Georgiou, Kuchta, Morayne and Niemiec~\cite{GKMN}, and gave an improved lower bound of $\frac{1}{4}$ for the probability of success.  More recently, Kozik~\cite{Koz} introduced a `dynamic threshold strategy' and showed that it was successful with probability at least $\frac{1}{4} + \epsilon$, for some $\epsilon > 0$ and for all sufficiently large posets.  Since the best possible probability of success in the classical secretary problem, on a totally ordered set, is $\frac{1}{e}$, the best possible lower bound for a universal algorithm must lie between $\frac{1}{4} + \epsilon$ and $\frac{1}{e}$.

In this paper, we show that given any poset there is an algorithm that is successful with probability at least $\frac{1}{e}$, so, in this sense, the total order is the hardest possible partial order. In fact, this algorithm depends only on the size of the poset and its number of maximal elements, so it is universal for any family where these are given. It is therefore natural to ask which is the hardest partial order with a given number of maximal elements. The most obvious choice is the poset consisting of $k$ disjoint chains. We shall give an asymptotically sharp lower bound on the probability of success in the problem of optimal choice on $k$ disjoint chains, and show that it is at least as hard as on any poset with $k$ maximal elements and of width $k$, that is, whose largest antichain has size $k$.

More precisely, our main aim is to prove the following two theorems.

\begin{theorem} \label{thm:pk}
Let $(P,\prec)$ be a poset with $k$ maximal elements and of width $k$.  Then there is an algorithm for the secretary problem on $(P,\prec)$ that is successful with probability at least $p_k$, where
\begin{equation} \label{eqn:pk}
p_k = \left\{ 
\begin{array} { c @{\quad \text{if } }l} \frac{1}{e}  & k=1 \\[+0.3ex]
\sqrt[k-1]{\frac{1}{k}} &  k > 1,
\end{array}\right.
\end{equation}
and these are the best possible such bounds.
\end{theorem}

We emphasize that in both the theorem above and that below, the claimed algorithm is not universal, but depends on both $|P|$ and the number of maximal elements of $(P,\prec)$.

\begin{theorem} \label{thm:1e}
Let $(P,\prec)$ be a poset.  Then there is an algorithm for the secretary problem on $(P,\prec)$ that is successful with probability at least $\frac{1}{e}$, and this is the best possible such bound.
\end{theorem}

We conjecture that Theorem~\ref{thm:pk} can be extended to all posets with $k$ maximal elements. It is not inconceivable that the same algorithm works; if not, it would be good to find some other algorithm dependent only on $k$ that does so.

\begin{conjecture}
Let $(P,\prec)$ be a poset with $k$ maximal elements.   Then there is an algorithm for the secretary problem on $(P,\prec)$ that is successful with probability at least $p_k$, where $p_k$ is as defined in~\eqref{eqn:pk}.
\end{conjecture}

Our algorithm, which gives the bound in Theorem~\ref{thm:1e}, depends only on size of the poset and the number of maximal elements. In the original version of this paper we conjectured that the latter piece of information is not needed; a beautiful proof of this result was given around the same time by Freij and W\"astlund~\cite{FW}. 

\begin{theorem}[Freij and W\"astlund~\cite{FW}]
There is an universal algorithm for the secretary problem which is successful on every poset $(P,\prec)$ with probability at least $\frac{1}{e}$.
\end{theorem}

We remark that the secretary problem on a poset with $k$ maximal elements was also considered recently (and independently of this work) by Kumar, Lattanzi, Vassilvitskii and Vattani~\cite{KLVV}, who obtained similar results via a different method. The poset consisting of $k$ disjoint chains was also studied by Kuchta and Morayne~\cite{KM}, but with a restriction on the order in which the elements are observed: those from the first chain all appear in a random order, then those from the second chain, and so on. This poset is also related to multicriteria extensions of the secretary problem. In the original multicriteria version, each element is ranked independently in $k > 1$ different criteria, and the selector wishes to select an element that is maximal in at least one of them.  This is equivalent to the problem on $k$ equally-sized disjoint chains with the elements appearing one at a time from each chain in the same cyclic order.  This version was solved by Gnedin~\cite{Gn81}.  Gnedin has also produced a more general survey of multicriteria problems~\cite{Gn92}.  Interestingly, the asymptotic value of the probability of success in Theorem~\ref{thm:pk}, $\sqrt[k-1]{\frac{1}{k}}$, is the same as in the multicriteria version.

This paper is organized as follows.  In Section~\ref{sec:formal}, we shall introduce the formal model and some notation.  In Section~\ref{sec:lower}, we shall describe a (randomized) algorithm for choosing an element of our poset, and prove lower bounds for its probability of success for various families of posets.  In Section~\ref{sec:upper}, we shall show that our bounds are best possible, by proving that, for the poset that consists of $k$ disjoint chains of length $n$ (which lies in each of these families), there is no strategy that wins with probability greater than $p_k + o(1)$ (as $n \to \infty$).

\section{Formal model and notation} \label{sec:formal}

We begin by defining formally the probability space in which we shall work throughout the paper.  The reader who wishes to avoid technicalities on a first reading is encouraged to skip this section, since all crucial definitions will be restated when used.

Our probability space will depend on a poset $(P,\prec)$ with $P = \{x_1,\ldots,x_n\}$.  Let $\max_{\prec}(P)$ denote the set of its maximal elements, that is,
\[
{\max}_{\prec}(P) = \{x \in P : \not\exists y \text{ such that } x \prec y\}.
\]
We shall suppress the subscript in $\max_{\prec}$ when it is clear from the context.

Given $(P,\prec)$, we shall work with a probability space $(\Omega_P, \Alg_P, \Prob_P)$, with $\Exp_P$ defined in the obvious way.  We shall suppress the subscripts when they are clear from the context, as they will be for the rest of this section.  We define the probability space $(\Omega, \Alg, \Prob)$ as follows.  Set $\Omega = S_n \times [0,1]$, where $S_n$ is the permutation group on $[n]$, and $\Alg = \Pow(S_n) \times \Borel$, where $\Borel$ is the Borel $\sigma$-algebra.  Let $\Prob = \mu \times \lambda$, where $\mu$ is the uniform probability measure, that is,
\[
\mu(\{\sigma\}) = \frac{1}{n!}
\]
for all $\sigma \in S_n$, and $\lambda$ is the Lebesgue measure.  In other words, $(\sigma,\delta) \in \Omega$ is picked uniformly at random.  Given $(\sigma,\delta) \in \Omega$, the $\sigma$-co-ordinate will determine the order in which elements of $P$ appear and the $\delta$-co-ordinate will allow us to introduce randomness independent of this order into our algorithms.  Specifically, the $\delta$-co-ordinate will determine an initial number of elements to reject without considering them.  The reason why we are using continuous space and Lebesgue measure, despite the fact that all of our randomized strategies pick one of a finite number of options, is that this allows them all to lie in the same probability space.

Write $P^{[n]}$ for the set of permutations of $P$, and let $\pi : \Omega \to P^{[n]}$ be the random variable defined by
\[
\pi(\sigma,\delta)(i) = x_{\sigma(i)}.
\]
Let $\Pos_t$ denote the set of all posets with vertex set $[t] = \{1,\ldots,t\}$.  Let $(P_t)_{t \in [n]}$ be a family of random variables with $P_t$ representing the poset we see at time $t$.  Formally, $P_t : \Omega \to \Pos_t$ and each $P_t(\sigma,\delta) = ([t],\prec_t)$ is defined by
\[
\forall i,j \in [t], i \prec_t j \Longleftrightarrow \pi(i) \prec \pi(j).
\]
The poset $P_t$ is the natural description of what we see at time $t$ as the elements of $P$ appear one by one.

Let $(\Alg_t)_{t \in [n]}$ be the sequence of $\sigma$-algebras with each $\Alg_t$ generated by the random variables $P_1,\ldots,P_t$, that is,
\[
\Alg_t = \sigma(P_1, \ldots, P_t) = \sigma(P_t),
\]
the second equality holding since $P_t$ is a labelled poset and so its value determines the values of $P_1,\ldots,P_{t-1}$.  We think of $\Alg_t$ as the information we know at time $t$ about where we are in the universe $\Omega$.  Since $P_t$ takes only finitely-many values, $\Alg_t$ has a simple structure; it is the pre-images in $\Omega$ of the possible values of $P_t$ and the unions of these pre-images.  We call these pre-images the atoms of $\Alg_t$.

Let $\Alg_t'$ be the projection of $\Alg_t$ onto $\Pow(S_n)$.  Since our definitions have so far depended only on the $\sigma$-co-ordinate of $(\sigma,\delta) \in \Omega$, we see that, for each $t$,
\[
\Alg_t = \{A \times [0,1] : A \in \Alg_t'\}.
\]
In other words, $(\sigma_1,\delta_1)$ and $(\sigma_2,\delta_2)$ are in the same atom of $\Alg_t$ if and only if $\sigma_1$ and $\sigma_2$ are in the same atom of $\Alg_t'$, which happens if and only if the labelled posets induced by the first $t$ elements $\pi(1),\ldots,\pi(t)$ are identical.

By a \emph{stopping time}, we mean a random variable $\tau$ taking values in $[n]$ and satisfying the property
\[
\{\tau = t\} \in \Alg_t,
\]
that is, our decision to stop at time $t$ is based only on the values of $P_1,\ldots,P_t$.

We shall need to refer to conditional expectation and probability, which in the finite world are trivial, intuitive concepts.  We define a family of random variables $(Z_t)_{t \in [n]}$ by
\[
Z_t = \Prob\big[ \pi(t) \in \max(P) \,|\, \Alg_t \big],
\]
that is, $Z_t$ is the probability that the $t$\th\ element observed is maximal given $P_1,\ldots,P_t$.  Our aim will be to choose a stopping time $\tau$ to maximize $\Prob\big[\pi(\tau) \in \max(P)\big]$.  The value of $\Prob\big[\pi(\tau) \in \max(P)\big]$ can be easily shown to be equal to $\Exp(Z_\tau)$ -- see page 45 of Chow, Robbins and Siegmund~\cite{CRS}, for example.  These equivalent formulations will be useful later.

Recall that $\Alg_t'$ is the projection of $\Alg_t$ onto $\Pow(S_n)$.  By a \emph{randomized stopping time}, we mean a random variable $\tau$ taking values in $[n]$ and satisfying the property
\[
\{\tau = t\} \in \Alg_t' \times \Borel,
\]
that is, our decision to stop at time $t$ is based on the values of $P_1,\ldots,P_t$ and on some $\Borel$-measurable random variable.  The randomized stopping times that we shall consider will be convex combinations of a finite number of true stopping times, so if such a randomized stopping time gives a certain probability of success, then there is a true stopping time with at least that probability of success.

\section{Lower bounds} \label{sec:lower}

Throughout this section, $p$ is a real number satisfying $0 < p < 1$.  Recall that $\pi(t)$ is the $t$\th\ element of the poset $P$ that we see, and that $P_t$ is a poset with vertex set $[t]$ that is isomorphic to the poset seen at time $t$.  We shall prove lower bounds for the probability of success of the following randomized algorithm on different families of posets.

\begin{alg}
Given a poset with $n$ elements, of which $k$ are maximal, let $X(p)\sim\text{Bin}(n,p)$.  Reject the first $X(p)$ elements and accept the first subsequent element where the following condition holds: the poset induced by the elements seen so far (including the currently observed element) has at most $k$ maximal elements and the currently observed element is one of them.
\end{alg}

This algorithm gives rise to the following stopping time, $\tau_k(p)$.

Let $X(p) : \Omega \to \{0,\ldots,n\}$ be the random variable defined by
\[
X(p)(\sigma,\delta) = \min \left\{x \geq 0 : \sum_{i=0}^x {n \choose i} p^i (1-p)^{n-i} \geq \delta\right\},
\]
so that
\[
\Prob(X(p)=x) = {n \choose x} p^x (1-p)^{n-x}
\]
and $X(p) = X(p)(\sigma,\delta)$ is independent of $\sigma$.  Then $\tau_k(p)$ is defined by
\[
\tau_k(p) = \left\{
\begin{array} { c @{\quad \text{} }l} 
\min\big\{t > X(p) : |\max(P_t)| \leq k \text{ and } t \in \max(P_t)\big\} & \text{if this exists},\\
	n & \text{otherwise}.
\end{array}\right.
\]

Given the definition of $\tau_k(p)$, it makes sense to consider another random variable, the set of $X(p)$ elements that we reject without considering.  We denote this random variable by $S(p)$, where
\[
S(p) = \{\pi(t) : t \leq X(p)\}.
\]

We shall make use of the following simple property of $S(p)$, which is easily verified.

\begin{lemma} \label{lem:independence}
The events $\{x \in S(p)\}_{x \in P}$ are independent and $\Prob(x \in S(p)) = p$ for all $x \in P$.
\end{lemma}

\begin{proof}
We can generate $\pi$ and $X(p)$ with the required distributions in the following way.  Put each element of $P$ in $S(p)$ with probability $p$ independently of all other elements.  Let $\pi$ consist of a uniformly random ordering of the elements of $S(p)$ followed by a uniformly random ordering of $P \setminus S(p)$.  By symmetry, $\pi$ is a uniformly random ordering of $P$, and $X(p) = |S(p)|$ is a binomial random variable independent of $\pi$.  The events $\{x \in S(p)\}_{x \in P}$ depend only on $\pi$ and $X(p)$, and by construction the properties in the statement of the lemma hold.
\end{proof}

We shall also use the following standard identity; for completeness we include a proof.

\begin{lemma} \label{lem:identity}
For all integers $k \geq 1$, the following holds:
\[
\sum_{s = 0}^{\infty} {k + s - 1 \choose k-1} (1-p)^s = \frac{1}{p^k}.
\]
\end{lemma}

\begin{proof}
Suppose that we have a coin that comes up heads with probability $p$ and tails with probability $1-p$, and that we toss it infinitely many times.  Then, with probability 1, we shall see at least $k$ heads, and the $k$\th\ head comes up in position $k+s$ for some $s \geq 0$.  In this case, we know that $k-1$ of the first $k+s-1$ tosses are heads and the remaining $s$ are tails, and so, summing over the probabilities that the $k$\th\ head comes in each position, we have
\[
\sum_{s = 0}^{\infty} {k + s - 1 \choose k-1} p^k (1-p)^s = 1,
\]
as required.
\end{proof}

In order to prove Theorem~\ref{thm:pk}, we first calculate a lower bound for the probability that $\tau_k(p)$ is successful on the poset consisting of $k$ disjoint chains.  Recall that $p$ is a real number satisfying $0 < p < 1$ and that $\pi(\tau_k(p))$ is the element that the algorithm $\tau_k(p)$ selects.

\setlength{\unitlength}{1pt}
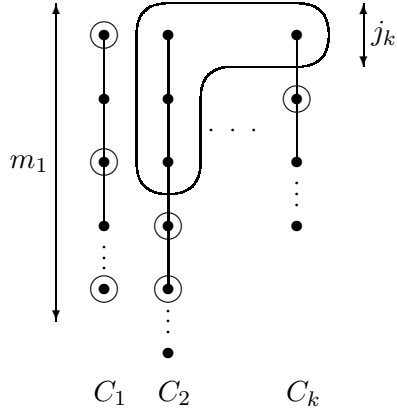
\begin{figure}[htp]
\begin{center}
\begin{picture}(126,150)(-15,-6)
%first chain
\multiput(12,36)(0,24){5}{\circle*{4}}
\put(12,60){\line(0,1){72}}
\multiput(12,44)(0,4){3}{\circle*{1}}
\put(12,36){\circle{10}}
\put(12,84){\circle{10}}
\put(12,132){\circle{10}}
%second chain
\multiput(36,12)(0,24){6}{\circle*{4}}
\put(36,36){\line(0,1){96}}
\multiput(36,20)(0,4){3}{\circle*{1}}
\put(36,36){\circle{10}}
\put(36,60){\circle{10}}
%ellipsis between chains
\multiput(52,96)(8,0){3}{\circle*{1}}
%third chain
\multiput(84,60)(0,24){4}{\circle*{4}}
\put(84,84){\line(0,1){48}}
\multiput(84,68)(0,4){3}{\circle*{1}}
\put(84,108){\circle{10}}
%circle available points
\qbezier(24,132)(24,144)(36,144)
\put(36,144){\line(1,0){48}}
\qbezier(84,144)(96,144)(96,132)
\qbezier(96,132)(96,120)(84,120)
\put(84,120){\line(-1,0){24}}
\qbezier(60,120)(48,120)(48,108)
\put(48,108){\line(0,-1){24}}
\qbezier(48,84)(48,72)(36,72)
\qbezier(36,72)(24,72)(24,84)
\put(24,84){\line(0,1){48}}
%add labels
\put(-6,84){\vector(0,1){60}}
\put(-6,84){\vector(0,-1){60}}
\put(109,132){\vector(0,1){12}}
\put(109,132){\vector(0,-1){12}}
\put(-15,82){\makebox[6pt][r]{$m_1$}}
\put(112,130){\makebox[6pt][l]{$j_k$}}
\put(8,-6){\makebox[8pt][l]{$C_1$}}
\put(32,-6){\makebox[8pt][l]{$C_2$}}
\put(80,-6){\makebox[8pt][l]{$C_k$}}
\end{picture}
\end{center}
\caption{An example of $k$ disjoint chains with the elements of $S(p)$ circled.  This illustrates an instance of the event $A_{0,3,\ldots,1}$.  The region enclosed by the solid curve marks the $j_1 + \ldots + j_k$ elements that might be selected.\label{fig:disjoint_chains}}
\end{figure}

\begin{theorem} \label{thm:lower_bound}
Let $(P,\prec)$ be a poset consisting of $k$ disjoint chains.  Then
\[
\Prob\big[ \pi(\tau_k(p)) \in \max(P) \big] >  \left\{
\begin{array} { c @{\quad \text{} }l} 
p \log \frac{1}{p} &\text{if } k = 1, \\
	\frac{k}{k-1} p (1-p^{k-1}) &\text{if } k > 1.
\end{array}\right.
\]
\end{theorem}

\begin{proof}
We first note that $\pi(\tau_k(p)) \in \max(P)$ in the exceptional case where $S(p) = P$ and $\pi(\tau_k(p)) = \pi(n) \in \max(P)$, an event with probability $\frac{k}{n}p^n$.  This tends to 0 as $n \to \infty$, and we obtain the bounds in the theorem by considering only the cases where $X(p) < n$ and hence $\pi(\tau_k(p)) \not\in S(p)$.  However, when we come to the proof of Lemma~\ref{lem:conditional_win}, the fact that these bounds are for a slightly smaller event will be important.

Let the $k$ chains be denoted by $C_1, \ldots, C_k$ and have lengths $m_1, \ldots , m_k$.  Let $A_{j_1, \ldots , j_k}$ be the event that for each $i$ there are $j_i$ elements from $C_i$ not in $S(p)$ above the highest element from $C_i$ in $S(p)$ (see Figure~\ref{fig:disjoint_chains}), that is,
\[
A_{j_1, \ldots , j_k} = \bigwedge_{i=1}^k \Big( \big| \big\{x \in C_i \backslash S(p) : \not\exists y \in C_i \cap S(p) \text{ such that } x \prec y \big\} \big| = j_i \Big).
\]
For $j_i < m_i$, this means that the top $j_i$ elements are not in $S(p)$ but the $(j_i + 1)$\st\ is.  For $j_i = m_i$, this means that there are no elements from the $i$\th\ chain in $S(p)$.  Note that if $A_{j_1, \ldots , j_k}$ occurs then $\pi(\tau_k(p))$ will be the first element observed from the $j_1 + \ldots + j_k$ elements not in $S(p)$ that are at the tops of their chains.

The events $\big\{ A_{j_1, \ldots , j_k} \,:\, 0 \leq j_1 \leq m_1, \ldots, 0 \leq j_k \leq m_k \big\}$ partition the whole space.  Thus, writing $Q_k(p)$ for $\Prob\big[\pi(\tau_k(p)) \in \max(P)\big]$,
\begin{align}
Q_k(p) & \, = \, \sum_{0 \leq j_1 \leq m_1, \ldots, 0 \leq j_k \leq m_k} \Prob\big[ \pi(\tau_k(p)) \in \max(P) \,|\, A_{j_1, \ldots , j_k} \big] \cdot \Prob\big[ A_{j_1, \ldots , j_k} \big] \notag \\
& \, > \, \sum_{\substack{0 \leq j_1 \leq m_1, \ldots, 0 \leq j_k \leq m_k \\ (j_1,\ldots,j_k) \neq (0,\ldots,0)}} \frac{|\{i:j_i > 0\}|}{j_1 + \ldots + j_k}(1-p)^{j_1 + \ldots + j_k}p^{|\{i:j_i < m_i\}|}. \label{eqn:raw_formula}
\end{align}
Since $1 + (1-p) + (1-p)^2 + \ldots = \frac{1}{p}$, this can be written as
\begin{eqnarray}
Q_k(p) & > & \sum_{\substack{ 0 \leq j_1 \leq m_1, \ldots, 0 \leq j_k \leq m_k \\ (j_1,\ldots,j_k) \neq (0,\ldots,0)}} \frac{ | \{ i:j_i > 0\} |}{ j_1 + \ldots + j_k} (1-p)^{j_1 + \ldots + j_k} p^k (1 + (1-p) + (1-p)^2 + \ldots)^{|\{i:j_i = m_i\}|} \nonumber \\
& = & \sum_{\substack{j_1, \ldots, j_k \geq 0 \\ (j_1,\ldots,j_k) \neq (0,\ldots,0)}} \frac{|\{i:j_i > 0\}|}{\min\{j_1,m_1\} + \ldots + \min\{j_k,m_k\}}(1-p)^{j_1 + \ldots + j_k}p^k \nonumber \\
& > & \sum_{\substack{j_1, \ldots, j_k \geq 0 \\ (j_1,\ldots,j_k) \neq (0,\ldots,0)}} \frac{|\{i:j_i > 0\}|}{j_1 + \ldots + j_k}(1-p)^{j_1 + \ldots + j_k}p^k. \label{eq:Q}
\end{eqnarray}
To see the equality, simply note that the term corresponding to $(j_1,\ldots,j_k)$ on the right-hand side appears on the left-hand side by choosing the term $(1-p)^{j_i-m_i}$ in the sum whenever $j_i \geq m_i$.

We now rewrite~\eqref{eq:Q} as a sum over $r = |\{i:j_i > 0\}|$ and $s = j_1 + \ldots + j_k$, and obtain
\[
Q_k(p) \, > \, \sum_{r = 1}^k \sum_{s = r}^\infty \Big| \Big\{(j_1,\ldots,j_k) : \big|\{i : j_i > 0\}\big| = r \text{ and } j_1+\ldots +j_k = s \Big\}\Big| \cdot  \frac{r}{s}(1-p)^sp^k.
\]

The rest of the proof is a straightforward calculation.  To calculate $\big|\big\{(j_1,\ldots,j_k) : \big|\{i : j_i > 0\}\big| = r \text{ and } j_1+\ldots +j_k = s\big\}\big|$, we note that there are ${k \choose r}$ ways of choosing the indices $i$ with $j_i > 0$ and there are then ${s-1 \choose r-1}$ ways for $r$ non-zero numbers to add up to $s$.  Thus
\[
Q_k(p) \; > \; \sum_{r = 1}^k \sum_{s = r}^\infty {k \choose r} {s - 1 \choose r - 1} \frac{r}{s} (1-p)^s p^k \; = \; kp^k \sum_{r = 1}^k \sum_{s = r}^\infty {k - 1 \choose r - 1} {s - 1 \choose r - 1} \frac{1}{s} (1-p)^s.
\]
Reversing the order of summation,
\[
Q_k(p) \; > \; kp^k \sum_{s = 1}^\infty \frac{1}{s} (1-p)^s \sum_{r = 1}^{\min\{k,s\}} {s - 1 \choose r - 1} {k - 1 \choose k - r}.
\]
The second sum is easily evaluated (a result known as Vandermonde's identity) to give
\begin{equation}
Q_k(p) > kp^k \sum_{s = 1}^\infty \frac{1}{s} (1-p)^s {k + s - 2 \choose k - 1}. \label{eqn:lower_bound}
\end{equation}

Finally, let us evaluate the sum in the above equation.  Write
\[
V_k(p) = \sum_{s = 1}^\infty \frac{1}{s} (1-p)^s {k + s - 2 \choose k - 1}.
\]
Differentiating, and then applying Lemma~\ref{lem:identity}, we find that
\[
\dd{V_k(p)}{p} \; = \;- \sum_{s = 1}^\infty (1-p)^{s-1} {k + s - 2 \choose k - 1} \;  = \; - \frac{1}{p^k}.
\]
We now integrate to obtain
\[
V_k(p) = \left\{
\begin{array} { c @{\quad \text{} }l} 
- \log{p} + c_1 &\text{if } k = 1, \\[+0.5ex]
	\frac{1}{(k-1)p^{k-1}} + c_k &\text{if } k > 1,
	\end{array}\right.
\]
where the $c_k$ are constants.  Since the expressions above are continuous in $p$ in the interval $(0,1]$, we may consider limits as $p \to 1$ to find $c_k$ and deduce that
\[
V_k(p) = \sum_{s = 1}^\infty \frac{1}{s} (1-p)^s {k + s - 2 \choose s - 1} = \left\{
\begin{array} { c @{\quad \text{} }l} 
\log \frac{1}{p} &\text{if } k = 1, \\
	\frac{1}{k-1}\left(\frac{1}{p^{k-1}} - 1\right) &\text{if } k > 1.	
	\end{array}\right.
\]
Substituting the value of $V_k(p)$ into \eqref{eqn:lower_bound} gives the result.
\end{proof}

In order to extend the result above to posets whose width is the same as their number of maximal elements, we shall use Dilworth's theorem~\cite{Dil} (see also page 81 of~\cite{MGT}):

\begin{dilworth}
A poset with largest antichain of size $k$ can be covered by $k$ chains.
\end{dilworth}

In the next theorem, we shall show that the secretary problem is no harder on a poset with $k$ maximal elements and width $k$ than on a poset consisting of $k$ disjoint chains.

\setlength{\unitlength}{1pt}
\begin{figure}[htp]
\begin{center}
\begin{picture}(126,150)(-15,-6)
%first chain
\multiput(12,36)(0,24){5}{\circle*{4}}
\put(12,60){\line(0,1){72}}
\multiput(12,44)(0,4){3}{\circle*{1}}
\put(12,36){\circle{10}}
\put(12,84){\circle{10}}
\put(12,132){\circle{10}}
%second chain
\multiput(36,12)(0,24){6}{\circle*{4}}
\put(36,36){\line(0,1){96}}
\multiput(36,20)(0,4){3}{\circle*{1}}
\put(36,36){\circle{10}}
\put(36,60){\circle{10}}
%ellipsis between chains
\multiput(52,96)(8,0){3}{\circle*{1}}
%third chain
\multiput(84,60)(0,24){4}{\circle*{4}}
\put(84,84){\line(0,1){48}}
\multiput(84,68)(0,4){3}{\circle*{1}}
\put(84,108){\circle{10}}
%extra link
\put(12,132){\line(1,-2){24}}
%circle available points
\qbezier(24,132)(24,144)(36,144)
\put(36,144){\line(1,0){48}}
\qbezier(84,144)(96,144)(96,132)
\qbezier(96,132)(96,120)(84,120)
\put(84,120){\line(-1,0){24}}
\qbezier(60,120)(48,120)(48,108)
\qbezier(48,108)(48,96)(36,96)
\qbezier(36,96)(24,96)(24,108)
\put(24,108){\line(0,1){24}}
%circle points no longer available
\multiput(24,84)(0,3){8}{\circle*{1}}
\multiput(48,84)(0,3){8}{\circle*{1}}
\put(25,81){\circle*{1}}
\put(26,78){\circle*{1}}
\put(28,76){\circle*{1}}
\put(30,74){\circle*{1}}
\put(33,73){\circle*{1}}
\put(36,72){\circle*{1}}
\put(39,73){\circle*{1}}
\put(42,74){\circle*{1}}
\put(44,76){\circle*{1}}
\put(46,78){\circle*{1}}
\put(47,81){\circle*{1}}
%add labels
\put(-6,84){\vector(0,1){60}}
\put(-6,84){\vector(0,-1){60}}
\put(109,132){\vector(0,1){12}}
\put(109,132){\vector(0,-1){12}}
\put(-15,82){\makebox[6pt][r]{$m_1$}}
\put(112,130){\makebox[6pt][l]{$j_k$}}
\put(8,-6){\makebox[8pt][l]{$C_1$}}
\put(32,-6){\makebox[8pt][l]{$C_2$}}
\put(80,-6){\makebox[8pt][l]{$C_k$}}
\end{picture}
\end{center}
\caption{An example of $k$ disjoint chains with one extra comparison, and with elements of $S(p)$ circled.  The region enclosed by the solid curve marks the elements that might be selected. The element in the dotted region could have been selected if the extra comparison were not there---cf. Figure~\ref{fig:disjoint_chains}.\label{fig:extra_comparison}}
\end{figure}

\begin{theorem} \label{thm:no_antichain}
Let $(P,\prec)$ be a poset with $n$ elements.  Suppose that $(P,\prec)$ has $k$ maximal elements and that none of its antichains has size greater than $k$.  Then
\[
\Prob\big[\pi(\tau_k(p)) \in \max(P)\big] > \left\{
\begin{array} { c @{\quad \text{} }l} 
p \log \frac{1}{p} &\text{if } k = 1, \\
	\frac{k}{k-1} p (1-p^{k-1}) &\text{if } k > 1.
	\end{array}\right.
\]
\end{theorem}

\begin{proof}
By Dilworth's theorem, we see that $P$ takes the form of $k$ chains with some comparisons in between them.  Clearly, the $k$ elements of $\max(P)$ lie at the top of the $k$ chains.  The proof therefore proceeds in an almost identical manner to that of Theorem~\ref{thm:lower_bound}.  The only difference is that the denominator in each term of \eqref{eqn:raw_formula} is now at most, rather than equal to, $j_1 + \ldots + j_k$ (see Figure~\ref{fig:extra_comparison}), so the expression in this line is still a lower bound.  The calculations that make up the remainder of the proof of Theorem~\ref{thm:lower_bound} therefore follow in the same way.
\end{proof}

The values that maximize the function in Theorem~\ref{thm:no_antichain} are
\begin{equation}  \label{eqn:p_k}
p_k = \left\{ 
\begin{array} { c @{\quad \text{if } }l} \frac{1}{e}  & k=1, \\[+0.3ex]
\sqrt[k-1]{\frac{1}{k}} &  k > 1. 
\end{array}\right.
\end{equation}
This gives us the following corollary and the lower bounds in Theorem~\ref{thm:pk}.

\begin{corollary} \label{cor:lower_bound}
Let $(P,\prec)$ be a poset with $n$ elements. Suppose that $(P,\prec)$ has $k$ maximal elements and that none of its antichains has size greater than $k$.  Then
\[
\Prob\big[ \pi(\tau_k(p_k)) \in \max(P) \big] > p_k.
\]
\end{corollary}
It is interesting to note that the expected proportion of elements that we reject without considering is the same as the probability of success.

We now wish to prove the following theorem, which, with the right choice of $p$, will give us a lower bound of $\frac{1}{e}$ for all posets, as in Theorem~\ref{thm:1e}.

\begin{theorem} \label{thm:known_max}
Let $(P,\prec)$ be a poset with $n$ elements.  Suppose that $(P,\prec)$ has $k$ maximal elements.  Then
\[
\Prob\big[\pi(\tau_k(p)) \in \max(P)\big] > kp^k\log \frac{1}{p}.
\]
\end{theorem}

The proof will use two simple lemmas.  The first states that the linear order is the hardest of all posets with a unique maximal element.

\begin{lemma} \label{lem:one_max_element}
Let $(P,\prec)$ be a poset with $n$ elements.  Suppose that $(P,\prec)$ has exactly one maximal element.  Then the probability that $\tau_1(p)$ is successful on $(P,\prec)$ is at least the probability that it is successful on a linear ordering of $P$, and hence
\[
\Prob\big[\pi(\tau_1(p)) \in \max(P)\big] > p \log \frac{1}{p}.
\]
\end{lemma}

\begin{proof}
We begin by taking an arbitrary linear extension of $\prec$, that is, a partial order $\prec'$ such that any two elements are comparable and such that $x \prec y \Rightarrow x \prec' y$.  (It is clear that such a partial order exists.)  We denote the unique element in $\max_{\prec}(P) = \max_{\prec'}(P)$ by $x_{\max}$.

Given this new poset, $(P,\prec')$, we define random variables $\pi'$, $X'(p)$, $S'(p)$ and $\tau_1'(p)$ in the same way as $\pi$, $X(p)$, $S(p)$ and $\tau_1(p)$ were defined given $(P,\prec)$.  We couple the random variables $(\pi, X(p), S(p), \tau_1(p))$ and $(\pi', X'(p), S'(p), \tau_1'(p))$ in the obvious way; we set $\pi' = \pi$ and $X'(p) = X(p)$, and hence $S'(p) = S(p)$.  This means that the elements appear in the same order in both instances, and the same set $S(p)$ is rejected in both cases.  The induced posets observed in the process on $(P,\prec')$ are linear extensions of those observed in the process on $(P,\prec)$.  We show that if $\pi(\tau_1(p)) \neq x_{\max}$ then $\pi'(\tau_1'(p)) \neq x_{\max}$, that is, if $\tau_1(p)$ fails in the process on $(P,\prec)$ then $\tau_1'(p)$ fails on $(P,\prec')$.  From this, the result follows, since the probability of success is therefore at least as large on $(P,\prec)$ as on $(P,\prec')$, and Theorem~\ref{thm:lower_bound} applied to $(P,\prec')$ gives the lower bound.

If we reach $x_{\max}$ then it will be accepted, since it must be the unique maximal element in the poset induced by the elements observed so far.  Thus $\pi(\tau_1(p)) \neq x_{\max}$ if either
\begin{enumerate}[(i)]
\item $x_{\max} \in S(p)$ or \label{case:max_gone}
\item after rejecting $S(p)$, we accept an element that appears earlier than $x_{\max}$. \label{case:false_max}
\end{enumerate}
In case~\eqref{case:max_gone}, $\tau_1'(p)$ must fail on $(P,\prec')$ for the same reason, since $S'(p) = S(p)$.  In case~\eqref{case:false_max}, such an element must be the unique maximal element of the poset induced by what we have seen so far, and this is still the case in any linear extension.  Therefore, with $\tau_1'(p)$, if this element is observed then it must be accepted, and so we still accept an element that appears earlier than $x_{\max}$.  It follows that in either case $\pi'(\tau_1'(p)) \neq x_{\max}$.
\end{proof}

The next lemma gives a lower bound for the probability of success restricting our attention to the case when all but one of the maximal elements of our poset are in $S(p)$.  This turns out to be enough to prove Theorem~\ref{thm:known_max}.

\begin{lemma} \label{lem:conditional_win}
Let $(P,\prec)$ be a poset with $n$ elements.  Suppose that $(P,\prec)$ has $k$ maximal elements.  Then
\[
\Prob\big[\pi(\tau_k(p)) \in \max(P) \,\big|\, |\max(P) \cap S(p)| = k - 1 \big] > \frac{p}{1-p}\log \frac{1}{p}.
\]
\end{lemma}

\begin{proof}
We first observe that we may assume that $k = 1$, for the following reason.  The condition that $|\max(P) \cap S(p)|=k-1$ means that the $k-1$ maximal elements in $\max(P) \cap S(p)$ will be maximal for the remainder of the process, so when using $\tau_k(p)$ we may ignore these and all elements dominated by at least one of these, and wait for a unique maximal element from the remaining elements.  Those elements form a poset $(P',\prec')$ with a unique maximal element $x_{\max}$, which is not in $S(p)$.  Since all elements are in $S(p)$ with probability $p$ independently of the others, the situation is the same as if we were working with $(P',\prec')$ and conditioning on $x_{\max} \not\in S(p)$.  Looking for one of at most $k$ maximal elements in $P$ using $\tau_k(p)$ is the same as looking for a unique maximal element in $P'$ using $\tau_1(p)$.

We assume from now on that $k=1$; we shall use Lemma~\ref{lem:one_max_element} to prove the result in this case.  Lemma~\ref{lem:one_max_element} used the bound from Theorem~\ref{thm:lower_bound}, and we recall from the proof of that theorem that the lower bound for $\Prob\big[\pi(\tau_1(p)) = x_{\max}\big]$ is in fact a lower bound for $\Prob\big[(S(p) \neq P) \wedge (\pi(\tau_1(p)) = x_{\max})\big]$. 

Let us write $M$ for the event that $x_{\max} \in S(p)$ and $W$ for the event $(S(p) \neq P) \wedge (\pi(\tau_1(p)) = x_{\max})$.  We note that if $S(p) \neq P$ and $x_{\max} \in S(p)$ then $\pi(\tau_1(p)) \neq x_{\max}$ and hence $W = W \wedge M^c$.  Therefore,
\[
\Prob(W) = \Prob(W \wedge M^c) = \Prob(M^c) \Prob(W \,|\, M^c) = (1-p)\Prob(W \,|\, M^c).
\]
Since, by the bound from Lemma~\ref{lem:one_max_element},
\[
\Prob(W) > p\log \frac{1}{p},
\]
and the quantity that we are interested in is $\Prob(W \,|\, M^c)$, the result follows.
\end{proof}

We are now in a position to prove our theorem.

\begin{proof}[Proof of Theorem~\ref{thm:known_max}]
We have
\[
\Prob\big[|\max(P) \cap S(p)|=k-1\big] = kp^{k-1}(1-p).
\]
Thus, for general $k$,
\begin{align*}
\Prob\big[\pi(\tau_k(p)) \in \max(P)\big] & \,>\, \Prob\Big[\pi(\tau_k(p)) \in \max(P) \,\big|\, |\max(P) \cap S(p)| = k - 1 \Big] \\
&\hspace{4cm} \mbox{} \times \Prob\Big[|\max(P) \cap S(p)|=k-1\Big] \\
&\,>\, \frac{p}{1-p}\log \frac{1}{p} \cdot kp^{k-1}(1-p) \; = \; kp^k\log \frac{1}{p},
\end{align*}
as required.
\end{proof}

This gives us the following corollary. The probability $e^{-\frac{1}{k}}$ is chosen to maximize the function in Theorem~\ref{thm:known_max} and gives the lower bound in Theorem~\ref{thm:1e}.  As mentioned earlier, it is well-known that $\frac{1}{e}$ is the best possible lower bound for the probability of success in the classical secretary problem, and so this completes Theorem~\ref{thm:1e}.

\begin{corollary} \label{cor:known_max}
Let $(P,\prec)$ be a poset with $n$ elements of which $k$ are maximal.  Then
\[
\Prob\Big[\pi\left(\tau_k\left(e^{-\frac{1}{k}}\right)\right) \in \max(P)\Big] > \frac{1}{e}.
\]
\end{corollary}

\section{Upper bound} \label{sec:upper}

In this section, we show that the bound in Corollary~\ref{cor:lower_bound} is best possible.  The proof of Theorem~\ref{thm:lower_bound} shows that the probability of success of the stopping time $\tau_k(p_k)$ on $k$ disjoint chains decreases towards the given lower bound as the chains increase in length.  This might suggest that the probability of success of an optimal strategy is reduced as the chains increase in length and thus, to prove that the bounds are best possible, we would consider chains with length tending to infinity.  The main theorem in this section, Theorem~\ref{thm:upper_bound}, does just that; for sufficiently long chains, the probability of success of an optimal stopping time can be made arbitrarily close to that in Corollary~\ref{cor:lower_bound}, and so $\tau_k(p_k)$ is asymptotically optimal.  Since the poset consisting of $k$ disjoint chains satisfies the conditions of Corollary~\ref{cor:lower_bound}, the bounds given are the best possible such bounds.

We define $D_k(x)$ to be the poset consisting of $k$ disjoint chains, each of size $x$.  It might be useful at this point to recall some definitions from Section~\ref{sec:formal}.  The probability space associated with the poset $D_k(x)$ is denoted by $(\Omega_{D_k(x)},\Alg_{D_k(x)},\Prob_{D_k(x)})$, but we suppress the subscripts when they are clear from the context.  The poset induced by the first $t$ elements that we observe is isomorphic to the random variable $P_t$, which is a poset on vertex set $[t]$, and $\Alg_t$ is the $\sigma$-algebra generated by $P_1,\ldots,P_t$, which represents what we know at time $t$.  A stopping time is a random variable $\tau$ taking values in $[n]$ and satisfying the property
\[
\{\tau = t\} \in \Alg_t.
\]
We shall use the notation $\Class_{(\Alg_t)}$ to denote the class of all such stopping times, and extend this notation to any sequence of $\sigma$-algebras in the analogous way.

We are trying to find an upper bound for $\Exp(Z_\tau)$ that holds for all stopping times $\tau$, where
\[
Z_t \,=\, \Prob\big[ \pi(t) \in \max(D_k(x)) \,|\, \Alg_t \big].
\]

Theorem \ref{thm:upper_bound} states that, as $x \to \infty$, the limit of the probability of success of the optimal stopping time on $D_k(x)$ is no greater than $p_k$.  Since Corollary \ref{cor:lower_bound} showed the existence of a stopping time with probability of success at least $p_k$, Theorem \ref{thm:upper_bound} shows that this is the best possible such bound and so gives Theorem \ref{thm:pk}.

\begin{theorem} \label{thm:upper_bound}
Let $p_k$ be as defined in~\eqref{eqn:pk}.  Then
\[
\lim_{x\to\infty} \sup_{\tau \in \Class_{(\Alg_t)}} \Exp_{D_k(x)}(Z_\tau) \leq p_k.
\]
\end{theorem}
Note that the supremum is over stopping times in $\Class_{(\Alg_t)}$, which means that we are allowed to use the extra information from the structure of the posets, not just the pay-offs that we are offered.

The following observation is important, so we record it as a lemma.

\begin{lemma} \label{lem:yx}
When $(P,\prec) = D_k(x)$, we have
\[
Z_t = \left\{
\begin{array} {c@{\quad \text{if } }l} 
y/x & \pi(t) \in \max(P_t), \pi(t) \in C \text{ and } |C \cap \{\pi(1),\ldots,\pi(t)\}| = y, \\[+1ex]
0 & \pi(t) \not\in \max(P_t),
\end{array}\right.
\]
where $C$ is one of the $k$ chains in $D_k(x)$.
\end{lemma}

\begin{proof}
The maximal element of a chain $C$ is equally likely to be at any position in the order in which its $x$ elements are observed.  Therefore, when $y$ elements have been observed from this chain, the probability that one of them is the maximal element is $\frac{y}{x}$, independent of the most recently observed element being maximal.
\end{proof}

The proof of Theorem \ref{thm:upper_bound} will proceed roughly as follows.  At time $t$ we expect to have seen approximately $\frac{t}{k}$ elements from each chain.  Therefore, since all orders are equally likely, the $t$\th\ element that we observe is maximal in what we have seen so far with probability approximately $\frac{1}{t/k} = \frac{k}{t}$.  By Lemma~\ref{lem:yx}, if this happens then it is a maximal element of $D_k(x)$ with probability approximately $\frac{t}{kx}$.  We conclude that $Z_t$ is approximately distributed as
\begin{equation}\label{eqn:Z_t_ditribution}
Z_t = \left\{
\begin{array} {c@{\quad \text{with probability } } l }
\frac{t}{kx} & \frac{k}{t}, \\[+0.5ex]
0 & 1 - \frac{k}{t}.
\end{array} \right.
\end{equation}

If $Z_t$ were distributed exactly like this with the $Z_t$ all independent of each other, then the proof would not be difficult to complete.  Since the potential non-zero value of $Z_t$ increases with $t$, it is straightforward to show (as in the classical secretary problem; more details will be given later in this section) that the optimal strategy is to ignore the first $I$ elements and accept the next non-zero $Z_t$.  Let us denote the associated stopping time by $\tau_I$ and make some rough calculations.  This is only an outline of the more precise argument that will be given later; in particular, $\approx$ is only intended to have an intuitive meaning and does not stand for any well-defined relation.  We find that
\begin{align*}
\Exp(Z_{\tau_I}) &= \sum_{t=I+1}^{kx} \Prob(\tau_I = t) \Exp(Z_t \,|\, \tau_I = t) \\
&\approx \sum_{t=I+1}^{kx} \Prob\Big[\big(Z_i = 0 \;\; \forall \, i \in \{I+1,\ldots,t-1\}\big) \wedge \big(Z_t > 0\big)\Big] \cdot \frac{t}{kx} \\
&\approx \sum_{t=I+1}^{kx} \left(\prod_{i=I+1}^{t-1}\left(1-\frac{k}{i}\right)\right) \cdot \frac{k}{t}\cdot\frac{t}{kx} \\
&=\frac{1}{x} \sum_{t=I+1}^{kx} \prod_{i=I+1}^{t-1}\left(1-\frac{k}{i}\right).
\end{align*}
We shall apply this formula in the case where $k$ is much smaller than $I$, so we can approximate $1-\frac{k}{i}$ by $e^{-\frac{k}{i}}$, and twice approximate sums by integrals to obtain \label{page:Z_t_calculations}
\begin{align*}
\Exp(Z_{\tau_I}) &\approx \frac{1}{x} \sum_{t=I+1}^{kx} e^{-\sum_{i=I+1}^{t-1} \frac{k}{i}} \; \approx \; \frac{1}{x} \sum_{t=I+1}^{kx} e^{-k\log\left(\frac{t}{I}\right)} \\
&= \frac{I^k}{x} \sum_{t=I+1}^{kx} t^{-k} \; \approx \;  \left\{
\begin{array} {c@{\quad \text{ } } l }
\frac{I}{x} \cdot \log\left(\frac{x}{I}\right) &\text{if } k = 1, \\[+0.5ex]
\frac{k}{k-1} \cdot \frac{I}{kx} \cdot \left(1 - \left(\frac{I}{kx}\right)^{k-1}\right) &\text{if } k > 1.
\end{array} \right.
\end{align*}

This is the formula in Theorem~\ref{thm:no_antichain} with $p = \frac{I}{kx}$ and is thus maximized, as in Corollary~\ref{cor:lower_bound}, when $\frac{I}{kx} = p_k$, in which case $\Exp(Z_{\tau_I}) \approx p_k$.  Therefore the bounds in Corollary~\ref{cor:lower_bound} are best possible, and if these calculations had been exact then the proof of Theorem~\ref{thm:pk} would be complete.

Unfortunately, $Z_t$ is not distributed exactly as in \eqref{eqn:Z_t_ditribution}.  In order to conclude that the optimal stopping time is of the simple form above, we should like to use the principle of \emph{backward induction}, described later (see also~\cite{CRS}, Theorem 3.2).  This formalizes the intuitive principle that, in a finite game, the optimal strategy is simply to analyse at each step whether or not we expect our situation to improve by continuing, and to do so if and only if this is the case.

The reason why the sequence of random variables $(Z_t)_{t \in [n]}$ is difficult to analyse is that the values they can take vary depending on how the process unfolds.  However, it is very likely that at any time we shall have seen approximately the same number of elements from each chain.  The proof will therefore proceed by defining a sequence of random variables $(Y_t)_{t \in [n]}$, which act as asymptotically almost sure upper bounds for $Z_t$ and are easier to analyse.  To obtain these bounds, we shall split each chain into $m$ segments, each of length $\ell$, and split the process into $m$ sets of $k\ell$ observations.  These lengths $\ell$ are margins of error beyond which we do not expect the number of elements observed from a chain to deviate.  Initially, we shall fix $m$ and let $\ell \to \infty$ to find an upper bound for $\Exp(Y_\tau)$ and hence $\Exp(Z_\tau)$ in terms of $m$.  Letting $m \to \infty$ will then give us a best possible result.

This means that the precise statement we shall prove for Theorem~\ref{thm:upper_bound} is in fact
\[
\lim_{m\to\infty} \lim_{\ell \to \infty} \sup_{\tau \in \Class_{(\Alg_t)}} \Exp_{D_k(\ell m)}(Z_\tau) \leq p_k.
\]
However, this is purely a matter of convenience; it is clear that the proof can be extended to posets $D_k(x)$ where $x$ is not a multiple of $m$ by dividing each chain into $m$ almost equal rather than exactly equal segments.

\setlength{\unitlength}{1pt}
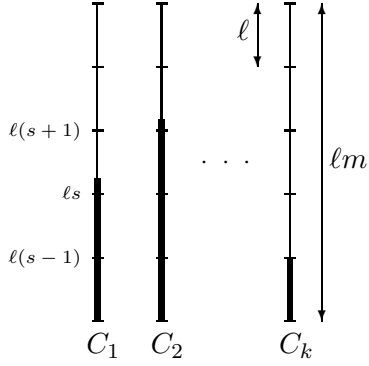
\begin{figure}[htp]
\begin{center}
\begin{picture}(96,132)
%first chain
\multiput(10,12)(0,24){6}{\line(1,0){4}}
\put(12,12){\line(0,1){120}}
%second chain
\multiput(34,12)(0,24){6}{\line(1,0){4}}
\put(36,12){\line(0,1){120}}
%ellipsis between chains
\multiput(52,72)(8,0){3}{\circle*{1}}
%third chain
\multiput(82,12)(0,24){6}{\line(1,0){4}}
\put(84,12){\line(0,1){120}}
%add labels
\put(96,72){\vector(0,1){60}}
\put(96,72){\vector(0,-1){60}}
\put(72,120){\vector(0,1){12}}
\put(72,120){\vector(0,-1){12}}
\put(99,70){\makebox[6pt][l]{$\ell m$}}
\put(63,118){\makebox[6pt][r]{$\ell$}}
\put(0,34){\makebox[6pt][r]{{\tiny{$\ell (s-1)$}}}}
\put(0,58){\makebox[6pt][r]{{\tiny{$\ell s$}}}}
\put(0,82){\makebox[6pt][r]{{\tiny{$\ell (s+1)$}}}}
\put(8,0){\makebox[8pt][l]{$C_1$}}
\put(32,0){\makebox[8pt][l]{$C_2$}}
\put(80,0){\makebox[8pt][l]{$C_k$}}
%add quantities seen
\linethickness{2pt}
\put(12,12){\line(0,1){54}}
\put(36,12){\line(0,1){76}}
\put(84,12){\line(0,1){24}}
\end{picture}
\end{center}
\caption{This figure shows the number of elements observed from each chain.  In this example, after a total of $k\ell s$ elements have been observed we see that $U_{C_1,s}$ and $U_{C_k,s}$ hold but $U_{C_2,s}$ does not.\label{fig:ucs}}
\end{figure}

We need to show that the process behaves in this approximately uniform manner with high probability as $\ell \to \infty$.  We shall first define what it means to be approximately uniform in one particular chain $C$ at time $k\ell s$, an event we call $U_{C,s}$ (see Figure~\ref{fig:ucs}), and then what it means to be approximately uniform everywhere at all times, an event we call $U$.

Given one of the chains, $C$, and for all $s \in \{0,\ldots,m\}$, let $U_{C, s}$ be the event that when we have observed $k\ell s$ elements in total we have observed between $\ell (s-1)$ and $\ell (s+1)$ elements from $C$, that is,
\[
U_{C,s} = \big\{ \ell (s-1) \leq |C \cap \{\pi(1),\ldots,\pi(k\ell s)\}| \leq \ell (s+1) \big\}.
\]
For all $t \in \{0,\ldots,k\ell m\}$, let $s(t)$ be the unique integer $s$ such that $k\ell (s-1) < t \leq k\ell s$, that is,
\begin{equation}\label{eq:sdef}
s(t) = \left\lceil\frac{t}{k\ell}\right\rceil. 
\end{equation}
Let $U$ be the event that for all $t$ when we have observed $t$ elements in total we have observed between $\ell (s(t)-2)$ and $\ell (s(t)+1)$ elements from each chain, that is,
\[
U = \bigcap_{i,t} \big\{\ell (s(t)-2) \leq |C_i \cap \{\pi(1),\ldots,\pi(t)\}| \leq \ell (s(t)+1) \big\}.
\]

We shall use Lemma~\ref{lem:uniformity}, which follows easily from Lemma~\ref{lem:UCs}.  It states that the process is approximately uniform with high probability.

\begin{lemma} \label{lem:UCs}
Let $m \geq 1$ be an integer, let $C$ be one of the chains in $D_k(\ell m)$ and let $s \in \{0,\ldots,m\}$.  Then
\[
\lim_{\ell \to \infty} \Prob_{D_k(\ell m)}(U_{C, s}) = 1.
\]
\end{lemma}

\begin{proof}
We show that the probability that we have observed more than $\ell (s+1)$ or fewer than $\ell (s-1)$ elements tends to zero as $\ell \to \infty$.  (If $s=0$ or $s=m$ then we need consider only one of these tails.)

Assume $C$ and $s$ are given.  Let $N$ be the number of elements we have observed from chain $C$ when we have observed $k\ell s$ elements in total.  It is straightforward to check that
\[
\Prob(N = x) = \frac{{\ell m \choose x}{(k-1)\ell m \choose k\ell s - x}}{{k\ell m \choose k\ell s}}
\]
is increasing for $x < \ell s$ and decreasing for $x > \ell s$, and that
\[
\left| \frac{ \Prob(N=x+1) }{ \Prob(N = x) } - 1 \right| > c > 0
\]
if $x \notin [\ell (s-1),\ell (s+1)]$, for some absolute $c > 0$.  It follows that $\Prob(N \notin [\ell (s-1),\ell (s+1)]) = O\left(\frac{1}{\ell}\right) \to 0$ as $\ell \to \infty$.
\end{proof}

\begin{lemma} \label{lem:uniformity}
Let $m \geq 1$ be an integer.  Then
\[
\lim_{\ell \to \infty} \Prob_{D_k(\ell m)}(U) = 1.
\]
\end{lemma}

\begin{proof}
This lemma follows simply from the previous lemma: choose $\ell$ sufficiently large that each of the $k(m+1)$ events $U_{C_i, s}$ occurs with probability at least $1 - \delta$.  Then, trivially, all $k(m+1)$ events hold with probability at least $1 - k(m+1)\delta$.  It is easy to see that
\[
\bigcap_{C_i,s} U_{C_i,s} \subseteq U,
\]
since the events $U_{C_i,s(t)-1}$ and $U_{C_i,s(t)}$ imply that
\[
\ell (s(t)-2) \leq |C_i \cap \{\pi(1),\ldots,\pi(t)\}| \leq \ell (s(t)+1),
\]
and so this holds for every $t$ and $i$, as required.
\end{proof}

The next lemma states that in order to prove Theorem~\ref{thm:upper_bound}, it suffices to show that its statement is true if we condition on $U$ holding.  This formalizes the intuition that, since the process is asymptotically almost surely uniform (that is, since $\lim_{\ell \to \infty} \Prob_{D_k(\ell m)}(U) \to 1$), we may assume that it is uniform.

Recall that $\Class_{(\Alg_t)}$ is the class of all stopping times relative to the $\sigma$-algebras $\Alg_t = \sigma(P_1,\ldots,P_t)$, that is, the decision to stop at time $t$ depends only on $P_1,\ldots,P_t$.

\begin{lemma} \label{lem:assume_uniform}
For all $m$,
\[
\lim_{\ell\to\infty} \sup_{\tau \in \Class_{(\Alg_t)}} \Exp_{D_k(\ell m)}(Z_\tau) \leq \lim_{\ell\to\infty} \sup_{\tau \in \Class_{(\Alg_t)}} \Exp_{D_k(\ell m)}(Z_\tau \,|\, U).
\]
\end{lemma}

\begin{proof}
By Lemma~\ref{lem:uniformity}, for all $\epsilon > 0$ we may choose $\ell$ sufficiently large that $\Prob_{D_k(\ell m)}(U^c) \leq \epsilon$.  We also know that $Z_t \leq 1$ for all $t$.  Therefore, for all $\tau$,
\begin{align*}
\Exp_{D_k(\ell m)}(Z_\tau) &= \Exp_{D_k(\ell m)}(Z_\tau \,|\, U)\Prob_{D_k(\ell m)}(U) + \Exp_{D_k(\ell m)}(Z_\tau \,|\, U^c)\Prob_{D_k(\ell m)}(U^c) \\
&\leq \Exp_{D_k(\ell m)}(Z_\tau \,|\, U) + \epsilon.
\end{align*}
We now take suprema to obtain
\[
\sup_{\tau \in \Class_{(\Alg_t)}} \Exp_{D_k(\ell m)}(Z_\tau) \leq \sup_{\tau \in \Class_{(\Alg_t)}} \Exp_{D_k(\ell m)}(Z_\tau \,|\, U) + \epsilon.
\]
Since $\epsilon$ is arbitrary, the result follows.
\end{proof}

Next, we define the random variables $Y_t$ that act as upper bounds for the $Z_t$ and are easier to analyse.  These random variables are not strict upper bounds, in the sense that the random variables are not coupled in any way.  However, conditioned on $U$ occurring, $Z_t$ is less than the potential non-zero value of $Y_t$ and the probability that $Z_t$ is non-zero is less than the probability that $Y_t$ is non-zero, and we shall be able to show that the optimal strategy for the game on $Z_t$ has a lower expected pay-off than the optimal strategy for the game on $Y_t$.

We shall often need to refer to the potential non-zero value of $Y_t$ and the probability that it takes this value, so let
\begin{equation}\label{eq:yp}
y_t = \frac{s(t)+1}{m} \qquad \text{and} \qquad \tilde{p}_t = \left\{
\begin{array} {c@{\quad \text{if } }l} 
\frac{1}{\ell (s(t)-2)} & s(t) \geq 3, \\[+1ex]
1 & s(t) \leq 2,
\end{array}\right.
\end{equation}
where $s(t)$ is as defined in~\eqref{eq:sdef}.  We now define a sequence of independent random variables $(Y_t)_{t \in [n]}$ by
\[
Y_t =  \left\{
\begin{array} {c@{\quad \text{ } } l }
y_t &\text{with probability } \tilde{p}_t, \\[+0.3ex]
0 &\text{with probability } 1 - \tilde{p}_t.
\end{array} \right.
\]
We shall not define these explicitly on any probability space as there is no need to do so, although it is of course straightforward to do it on $\Omega$.

The next lemma states the intuitive principle that we expect to do at least as well in the game with the random variables $Y_t$ as in the game with the random variables $Z_t$ conditioned on $U$ occurring.  Analagously to $\Alg_t$ for $Z_t$, let $(\Algg_t)_{t \in [n]}$ be defined by $\Algg_t = \sigma(Y_1, \ldots, Y_t)$.

\begin{lemma} \label{lem:upper}
For all $\ell,m \in \mathbb{N}$ with $m \ge 3$,
\[
\sup_{\tau \in \Class_{(\Alg_t)}} \Exp_{D_k(\ell m)}(Z_\tau \,|\, U) < \sup_{\tau \in \Class_{(\Algg_t)}} \Exp_{D_k(\ell m)}(Y_\tau).
\]
\end{lemma}

Putting Lemmas~\ref{lem:assume_uniform} and~\ref{lem:upper} together tells us that
\[
\lim_{\ell\to\infty} \sup_{\tau \in \Class_{(\Alg_t)}} \Exp_{D_k(\ell m)}(Z_\tau) \leq \lim_{\ell\to\infty} \sup_{\tau \in \Class_{(\Algg_t)}} \Exp_{D_k(\ell m)}(Y_\tau).
\]

In order to prove this lemma, we need a precise statement of what backward induction tells us is the optimal stopping time in a finite process.  We define a new random variable for each $t$, the \emph{value} of the game at time $t$.  This is the expected pay-off ultimately accepted given what has happened so far.  We calculate these values inductively, starting at the end.  The value of the game at the final step is just the final pay-off offered.  The value of the game at each earlier step is the maximum of the currently offered pay-off and the expected value of the game at the next step.  The optimal strategy is to stop when the currently offered pay-off is at least the expected value at the next step.

In the backward induction theorem below, the pay-offs offered are the $W_t$ and the values at each step are the $\gamma_t$.  The $\sigma$-algebras $\AlgA_t$ represent what we know at time $t$.  We remind the reader that being $\AlgA_t$-measurable means that $\sigma(W_t) \subset \AlgA_t$, that is, the value of $W_t$ is determined by what we know at time $t$ or, in the finite world, $W_t$ is constant on each atom of $\AlgA_t$.  In fact, the nested condition means that $\AlgA_t \supset \sigma(W_1,\ldots,W_t)$.  The statement of the theorem is that the strategy that stops at the first $t$ when $W_t = \gamma_t$ (or, equivalently, when $W_t$ is at least as large as the expected value of $\gamma_{t+1}$ given $\AlgA_t$) is indeed a stopping time and achieves the optimal value.  For more details, see~\cite[Theorem 3.2]{CRS}.

\begin{backward_induction}
Let $\AlgA_1 \subset \ldots \subset \AlgA_n$ be a nested sequence of $\sigma$-algebras and let $W_1, \ldots, W_n$ be a sequence of random variables with each $W_t$ being $\AlgA_t$-measurable.  Let $\Class_{(\AlgA_t)}$ be the class of stopping times relative to $(\AlgA_t)_{t \in [n]}$ and let $v^*$ be given by
\[
v^* = \sup_{\tau \in \Class_{(\AlgA_t)}} \Exp(W_\tau).
\]
Define successively $\gamma_n, \gamma_{n-1}, \ldots, \gamma_1$ by setting
\begin{align*}
\gamma_n &= W_n, \\
\gamma_t &= \max\big\{W_t, \,\Exp(\gamma_{t+1} \,|\, \AlgA_t) \big\}, \quad t = n-1, \ldots, 1.
\end{align*}
Let
\[
\tau^* = \min\{t : W_t = \gamma_t\}.
\]
Then $\tau^* \in \Class_{(\AlgA_t)}$ and
\[
\Exp(W_{\tau^*}) = \Exp(\gamma_1) = v^* \geq \Exp(W_\tau) \text{ for all } \tau \in \Class_{(\AlgA_t)}.
\]
\end{backward_induction}

This theorem provides the machinery we need to prove our lemma.

\begin{proof}[Proof of Lemma~\ref{lem:upper}]
We shall apply backward induction to the sequences $Y_t$ and $Z_t$, conditioned on $U$ occurring, to show that the optimal expected pay-off for the $Y_t$ is at least as large as for the $Z_t$.  For convenience, we shall continue to use $n$ to denote the number of elements in $D_k(\ell m)$, that is, $n = k\ell m$.

Let us first consider what happens with the sequence $Y_t$.  Recall that $(\Algg_t)_{t \in [n]}$ are defined by $\Algg_t = \sigma(Y_1, \ldots, Y_t)$ and let $(\alpha_t)_{t \in [n]}$ be defined for $(Y_t)_{t \in [n]}$ as $(\gamma_t)_{t \in [n]}$ were for $(W_t)_{t \in [n]}$ in the backward induction theorem, that is,
\begin{align*}
\alpha_n & = Y_n, \\
\alpha_t &= \max\big\{ Y_t, \Exp\big( \alpha_{t+1} \,|\, \Algg_t \big) \big\}, \quad t = n-1, \ldots, 1.
\end{align*}
Since the random variables $Y_t$ are independent, the values of $Y_1,\ldots,Y_t$ give no information about the values of $Y_{t+1},\ldots,Y_n$, and therefore $\Exp(\alpha_{t+1} \,|\, \Algg_t)$ is constant on $\Algg_t$ and equal to $\Exp(\alpha_{t+1})$.  Therefore we may define the function $v : [n] \to \Real$ by
\[
v(t) = \Exp(\alpha_t)
\]
and note that backward induction tells us that the stopping time that stops at the first $t$ such that $Y_t \geq \Exp\big( \alpha_{t+1} \,|\, \Algg_t \big) = v(t+1)$ is optimal.  By definition,
\[
v(t) = \Exp\big( \max\{Y_t, v(t+1)\} \big) \geq v(t+1),
\]
whereas $y_t$, the potential non-zero value of $Y_t$, is a non-decreasing function of $t$.  We conclude that there exists $I$ such that
\begin{equation}\label{eq:Idef}
\begin{array} {c@{\quad \text{if } }l} 
y_t < v(t+1) \quad & t \leq I, \\[+0.5ex]
y_t \geq v(t+1) \quad &  t > I,
\end{array}
\end{equation}
and therefore an optimal strategy for the game on $Y_t$ is `reject the first $I$ elements, and accept the next with a non-zero pay-off.'

Recall that the distribution of $Y_t$ is given by
\[
Y_t =  \left\{
\begin{array} {c@{\quad \text{ } } l }
y_t &\text{with probability } \tilde{p}_t, \\[+0.3ex]
0 &\text{with probability } 1 - \tilde{p}_t.
\end{array} \right.
\]
We deduce that, since $m \geq 3$ and $s(n) = m$,
\begin{equation}\label{eq:vn}
v(n) \,=\, \Exp(\alpha_n) \,=\, \Exp(Y_n) \,=\, \tilde{p}_ny_n \,=\, \frac{m+1}{m} \cdot \frac{1}{\ell (m-2)} \,>\, \frac{1}{\ell m} \,=\, \frac{k}{n}
\end{equation}
and that, for $1 \leq t \leq n-1$,
\begin{equation}\label{eq:vrec}
v(t) \,=\, \Exp(\alpha_t) \,=\, \Exp\big( \max\{Y_t, \,\Exp(\alpha_{t+1})\} \big) \,=\,  \left\{
\begin{array} {c@{\quad \text{ } } l }
\tilde{p}_t y_t + (1-\tilde{p}_t)v(t+1) &\text{if } t > I, \\[+0.3ex]
v(t+1) &\text{if } t \leq I.
\end{array} \right.
\end{equation}

We now turn our attention to the sequence of random variables $Z_t$.  Since we are conditioning on $U$, we introduce a new sequence of $\sigma$-algebras $\Alggg_t$ defined by
\[
\Alggg_t = \sigma(\Alg_t \cup \{U\})
\]
and shall consider only $\omega \in U$.  Analogously to $\gamma_t$ and $\alpha_t$, let
\[
\beta_n = Z_n
\]
and, for $1 \leq t \leq n-1$, let
\[
\beta_t = \max\big\{ Z_t, \Exp\big( \beta_{t+1} \,|\,  \Alggg_t \big) \big\}.
\]
Recalling that $\Class_{(\Alg_t)}$ denotes the class of stopping times relative to $\Alg_t$, observe that we have $\Alg_t \subset \Alggg_t$, and hence $\Class_{(\Alg_t)} \subset \Class_{(\Alggg_t)}$ and
\[
\sup_{\tau \in \Class_{(\Alg_t)}} \Exp_{D_k(\ell m)}(Z_\tau \,|\, U) \leq \sup_{\tau \in \Class_{(\Alggg_t)}} \Exp_{D_k(\ell m)}(Z_\tau \,|\, U).
\]
Note that intuitively this is obvious: it just says that having extra information (that the event $U$ occurs) can only help us in choosing our stopping time $\tau$.

Recall that $\Exp(\beta_t \,|\, \Alggg_{t-1})(\omega)$ denotes the expected value of the game at time $t$, if we have so far seen the first $t-1$ elements of $P$ and are told whether or not $U$ holds (that is, whether or not $\omega \in U$). The lemma follows from the following claim by the Backward Induction Theorem.

\begin{claim*}
For all $\omega \in U$ and for all $t \in [n]$,
\[
\Exp\big( \beta_t \,|\, \Alggg_{t-1} \big)(\omega) < v(t),
\]
where $\Alggg_0 = \{\emptyset,U,U^c,\Omega\}$ and so $\Exp(\beta_1 \,|\, \Alggg_0)(\omega) = \Exp(\beta_1 \,|\, U)$.
\end{claim*}

\begin{proof}[Proof of claim]
We shall prove the claim by induction on $n - t$, using~\eqref{eq:vn} and~\eqref{eq:vrec}. First, recall that $\Exp( \beta_{n} \,|\, \Alggg_{n-1})(\omega)$ is just the probability that the final element of $\omega$ is maximal in $P$. Thus, by~\eqref{eq:vn},
\[
\Exp\big( \beta_{n} \,|\, \Alggg_{n-1} \big) = \frac{k}{n} < v(n),
\]
which proves the case $n - t = 0$.

So let $1 \leq t \leq n-1$, and assume that the result holds for $t + 1$. We claim that
\begin{equation}\label{eq:beta}
\Exp\big( \beta_t \,\big|\,\Alggg_{t-1} \big)(\omega) \, < \,  \left\{
\begin{array} {c@{\quad \text{ } } l }
\tilde{p}_t y_t + (1-\tilde{p}_t)v(t+1) &\text{if } t > I, \\[+0.3ex]
v(t+1) &\text{if } t \leq I,
\end{array} \right.
\end{equation}
and hence that $\Exp\big( \beta_t \,\big|\,\Alggg_{t-1} \big)(\omega) < v(t)$, by~\eqref{eq:vrec}. 

In order to prove~\eqref{eq:beta}, let $\omega \in U$ and consider the atom $A \in \Alggg_{t-1}$ which contains $\omega$. Partitioning the space according to whether or not $Z_t \geq \Exp \big(\beta_{t+1} \,|\, \Alggg_t \big)$, we obtain
\begin{multline}\label{eq:C1}
\Exp\big( \beta_t \,|\, \Alggg_{t-1} \big)(\omega) \, = \, \Exp\big(\beta_t \,|\, A \big) \, = \, \Prob\Big( Z_t \geq \Exp \big(\beta_{t+1} \,|\, \Alggg_t \big) \,\big|\, A \Big) \cdot \Exp\Big( Z_t \,\big|\, \big( Z_t \geq \Exp(\beta_{t+1} \,|\, \Alggg_t) \big) \cap A\Big) \\
+ \, \Prob\Big( Z_t < \Exp \big(\beta_{t+1} \,|\, \Alggg_t \big) \,\big|\, A\Big) \cdot \Exp\Big(\Exp \big(\beta_{t+1} \,|\, \Alggg_t \big) \,\big|\, \big(Z_t < \Exp \big(\beta_{t+1} \,|\, \Alggg_t) \big) \cap A \Big),
\end{multline}
since if $Z_t \geq \Exp \big(\beta_{t+1} \,|\, \Alggg_t \big)$ then the payoff is $Z_t$, and otherwise the payoff is $\Exp \big(\beta_{t+1} \,|\, \Alggg_t \big)$.

Now, by the induction hypothesis we have 
$$\Exp\big(\beta_{t+1} \,|\, \Alggg_t\big)(\omega') < v(t+1)$$ 
for every $\omega' \in (Z_t < \Exp(\beta_{t+1} \,|\, \Alggg_t)) \cap A$, and therefore
\begin{equation}\label{eq:C2}
\Exp\Big( \Exp \big(\beta_{t+1} \,|\, \Alggg_t \big) \,\big|\, \big( Z_t < \Exp(\beta_{t+1} \,|\, \Alggg_t) \big) \cap A \Big)  <  v(t+1).
\end{equation}
Moreover, by Lemma~\ref{lem:yx} and~\eqref{eq:yp}, and since $\omega \in U$, we have
\[
Z_t(\omega) \leq y_t \qquad \text{and} \qquad \Prob\big( Z_t > 0 \,|\, \Alggg_{t-1} \big)(\omega) \leq \tilde{p}_t.
\]
Hence
\begin{equation}\label{eq:C3}
\Exp\Big( Z_t \,\big|\, \big( Z_t \geq \Exp(\beta_{t+1} \,|\, \Alggg_t) \big) \cap A \Big) \le y_t,
\end{equation}
and
\begin{equation}\label{eq:C4}
\Prob\big( Z_t < \Exp(\beta_{t+1} \,|\, \Alggg_t) \,\big|\, A\big) \geq 1 - \tilde{p}_t.
\end{equation}
Finally, recall that $I$ was chosen so that $y_t < v(t+1)$ if and only if $t \leq I$. Now~\eqref{eq:beta} follows easily from~\eqref{eq:C1},~\eqref{eq:C2},~\eqref{eq:C3} and~\eqref{eq:C4}. This completes the induction step, and hence proves the claim.
\end{proof}

\begin{comment}
% I removed these comments from the proof
$(Z_t < \Exp(\beta_{t+1} \,|\, \Alggg_t)) \in \Alggg_t$ and $\Alggg_{t-1} \subset \Alggg_t$, and so $\big( Z_t < \Exp(\beta_{t+1} \,|\, \Alggg_t) \big) \cap A \in \Alggg_t$.  Note first that, since $\omega \in U$ and $U \in \Alggg_{t-1}$, we must have $A \subset U$, and second that $\big( Z_t < \Exp(\beta_{t+1} \,|\, \Alggg_t) \big) \cap A$ is non-empty since $\Prob\big( Z_t < \Exp(\beta_{t+1} \,|\, \Alggg_t) \,\big|\, A\big) \geq 1-\tilde{p}_t > 0$.  

Since $y_t \le v(t+1)$ if and only if $t \le I$, and $\Prob\big( Z_t \geq \Exp(\beta_{t+1} \,|\, \Alggg_t) \,\big|\, A \big) \leq \Prob(Z_t > 0 \,|\, A) \leq \tilde{p}_t < 1$, it follows that 
$$\Exp\big( \beta_t \,\big|\,\Alggg_{t-1} \big)(\omega) \, < \, \left\{
\begin{array} {c@{\quad \text{ } } l }
\tilde{p}_t y_t + (1-\tilde{p}_t)v(t+1) &\text{if } t > I, \\
v(t+1) &\text{if } t \leq I,
\end{array} \right.$$
and hence that $\Exp\big( \beta_t \,\big|\,\Alggg_{t-1} \big)(\omega) < v(t)$, by the observations above.
\end{comment}

The lemma follows from the claim, since, by the backward induction theorem, we have
\[
\sup_{\tau \in \Class_{(\Alggg_t)}} \Exp_{D_k(\ell m)}(Z_\tau \,|\, U) = \Exp(\beta_1 \,|\, U) < v(1) = \Exp(\alpha_1) = \sup_{\tau \in \Class_{(\Algg_t)}} \Exp_{D_k(\ell m)}(Y_\tau),
\]
as required.
\end{proof}

In the final lemma before the proof of Theorem~\ref{thm:upper_bound}, we use backward induction to show that an optimal stopping time for the process with the $Y_t$ takes the simple form of rejecting the first $k\ell u^*$ elements, for some integer $u^*$, and accepting the next non-zero pay-off.

\begin{lemma} \label{lem:simple_stopping_times}
For $u^* \in \{0,\ldots,m-1\}$, let
\[
\tau_{u^*} = \left\{
\begin{array} {c@{\quad \text{} }l} 
\min\left\{t > k\ell u^* : Y_t > 0\right\} & \text{if this exists},\\
	n & \text{otherwise}.
\end{array}\right.
\]
Then
\[
\sup_{\tau \in \Class_{(\Algg_t)}} \Exp_{D_k(\ell m)}(Y_\tau) = \sup_{u^* \in \{0,\ldots,m-1\}} \Exp_{D_k(\ell m)}(Y_{\tau_{u^*}})
\]
\end{lemma}

\begin{proof}
We have already shown in the proof of Lemma~\ref{lem:upper} that an optimal strategy takes the form `ignore the first $I$, and accept the next non-zero pay-off.'  In fact, we see that $I$ must be a multiple of $k\ell$: suppose, for contradiction, that $I = k\ell u^* + r$, where $r \in \{1,\ldots, k\ell - 1\}$. Then $s(I) = s(I+1)$ by~\eqref{eq:sdef}; recall from~\eqref{eq:Idef} that $y_t = (s(t)+1) / m < v(t+1)$ if and only if $t \leq I$. It follows that
\[
v(I+2) \leq \frac{s(I+1)+1}{m}  = \frac{u^*+2}{m},
\]
since we would be willing to stop at time $I+1$. But then, by~\eqref{eq:vrec},
\[
v(I+1) \leq \max\big\{ y_{I+1}, v(I+2)\big\} = \frac{u^*+2}{m} = \frac{s(I)+1}{m}.
\]
Thus we would in fact be willing to stop at time $I$, which is a contradiction.
\end{proof}

We are now in a position to complete the proof of the main theorem in this section.

\begin{proof}[Proof of Theorem~\ref{thm:upper_bound}]
All that remains is to calculate and maximize $\Exp(Y_{\tau_{u^*}})$ over $u^*$, where $\tau_{u^*}$ is the smallest $t > k\ell u^*$ such that $Y_t > 0$.  These calculations are very similar to those on page \pageref{page:Z_t_calculations}, but also include error terms which tend to zero as $\ell,m \to \infty$.

We shall assume first that $u^* \to \infty$ as $m \to \infty$, and deduce from our calculation that this is a valid assumption. Indeed, if $t = o(n)$ then $y_t = o(1)$, whereas we shall obtain a probability of success that is separated from zero. We should therefore never accept a payoff for $t = o(n)$. In particular, for sufficiently large $m$ we have $u^* \ge 3$, and so
\[
\Exp\left(Y_{\tau_{u^*}}\right) = \sum_{u = u^*}^{m-1} \sum_{r=1}^{k\ell} \Prob\big(Y_{k\ell u^*+1}=0,\ldots,Y_{k\ell u+r-1}=0,Y_{k\ell u+r} > 0 \big)y_{k\ell u+r}.
\]
Recall that the $Y_t$ are independent, and that $s(k\ell u+h) = u+1$ when $1 \leq h \leq k\ell$. Hence, using the convention that the empty product takes the value 1, we obtain
\begin{align*}
\Exp\left(Y_{\tau_{u^*}}\right) & = \sum_{u = u^*}^{m-1} \sum_{r=1}^{k\ell} \left(\prod_{q = u^*+1}^u \left(1-\frac{1}{\ell (q - 2)}\right)^{k\ell}\right) \left(1-\frac{1}{\ell (u-1)}\right)^{r-1} \cdot \frac{1}{\ell (u-1)} \cdot \frac{u+2}{m} \\
&= \sum_{u = u^*}^{m-1} \left(\prod_{q = u^*+1}^u \left(1-\frac{1}{\ell (q - 2)}\right)^{k\ell}\right) \left(1 - \left(1-\frac{1}{\ell (u-1) }\right)^{k\ell}\right) \cdot \frac{u+2}{m} \\
&= \sum_{u = u^*}^{m-1} \left(\prod_{q = u^*+1}^u \left(1-\frac{1}{\ell (q - 2)}\right)^{k\ell}\right) \cdot \frac{u+2}{m} \\
&\hspace{4cm} \mbox{} - \sum_{u = u^*+1}^m \left(\prod_{q = u^*+1}^u \left(1-\frac{1}{\ell (q - 2)}\right)^{k\ell}\right) \cdot \frac{u+1}{m} \\
&= \frac{u^*+2}{m} - \left(\prod_{q = u^*+1}^{m} \left(1-\frac{1}{\ell (q - 2)}\right)^{k\ell}\right) \cdot \frac{m+1}{m} \\
&\hspace{4cm} \mbox{} + \sum_{u = u^*+1}^{m-1} \left(\prod_{q = u^*+1}^u \left(1-\frac{1}{\ell (q - 2)}\right)^{k\ell}\right) \cdot \frac{1}{m}. \\
\end{align*}

Now, let $\epsilon > 0$ be arbitrary, choose $m = m(\epsilon)$ and $\ell = \ell (m,\epsilon)$ sufficiently large, and recall that therefore $u^* = u^*(\epsilon)$ may be chosen to be sufficiently large also. Since $\left(1 - \frac{1}{n}\right)^n < \frac{1}{e} < \left(1 - \frac{1}{n}\right)^{n-1}$ for all $n \geq 2$, we have
$$\prod_{q = u^*+1}^{m} \left(1-\frac{1}{\ell (q - 2)}\right)^{k\ell} \; \ge \; \exp\left( - \left( \frac{k\ell + 1}{\ell} \right) \sum_{q = u^*+1}^{m} \frac{1}{q-2} \right) \; \ge \; \left( \frac{u^*}{m} \right)^k - \frac{\epsilon}{2},$$
and similarly
$$\prod_{q = u^*+1}^u \left(1-\frac{1}{\ell (q - 2)}\right)^{k\ell} \; \le \; \exp\left( - k \sum_{q = u^*+1}^{u} \frac{1}{q-2} \right) \; \le \; \left( \frac{u^*}{u} \right)^k.$$
Setting $p = \frac{u^*}{m}$, we obtain
\begin{align*}
\Exp\left(Y_{\tau_{u^*}}\right) & \leq \left\{
\begin{array} {c@{\quad \text{ } } l }
p-p + p \log \frac{1}{p}  + \epsilon &\text{if } k = 1, \\
p - p^k + \displaystyle\frac{p}{k-1} \left(1 - p^{k-1} \right) + \epsilon &\text{if } k > 1.
\end{array} \right.
\end{align*}

As before, these expressions are maximized when $p = p_k$, and thus
$$\lim_{m \to \infty} \lim_{\ell \to \infty} \Exp\big( Y_{\tau_{u^*}}  \big) \; \le \; p_k + \epsilon.$$
Since $\epsilon > 0$ was arbitrary, this completes the proof.
\end{proof}

Putting Corollary~\ref{cor:lower_bound} and Theorem~\ref{thm:upper_bound} together gives Theorem~\ref{thm:pk}.

\section*{Acknowledgements}

We should like to thank B\'ela Bollob\'as and Graham Brightwell for their several significant contributions to this paper. In particular, we are grateful for important ideas relating to an earlier proof of Theorem~\ref{thm:1e}, for further productive discussions about the problem, and for their comments on the paper.

\end{document}